\documentclass[12pt,a4paper]{article}

\usepackage{amssymb,amsmath,amsthm}
\usepackage{graphicx}
\usepackage{hyperref}

\newcommand\cC{{\mathcal C}}
\newcommand\cD{{\mathcal D}}

\newcommand\cF{{\mathcal F}}

\newcommand\bx{\mathbf x}

\newcommand\cG{{\mathcal G}}

\newcommand\cI{{\mathcal I}}

\newcommand\cM{{\mathcal M}}

\newcommand\cP{{\mathcal P}}

\newcommand\cU{{\mathcal U}}

\theoremstyle{plain}
\newtheorem{theorem}{Theorem}[section]
\newtheorem{lemma}[theorem]{Lemma}

\newtheorem{conjecture}[theorem]{Conjecture}
\newtheorem{proposition}[theorem]{Proposition}

\newtheorem{problem}[theorem]{Problem}
\theoremstyle{definition}

\newtheorem{construction}[theorem]{Construction}
\newtheorem{claim}[theorem]{Claim}

\newcommand\cref[1]{Corollary~\ref{cor:#1}}

\title{On colorings of the Boolean lattice avoiding a rainbow copy of a poset}
\begin{document}

\author{Bal\'azs Patk\'os
\\
\small Alfr\'ed R\'enyi Institute of Mathematics, Hungarian Academy of Scinces \\
\small H-1053, Budapest, Re\'altanoda u. 13-15}

\maketitle

\begin{abstract}
    Let $F(n,k)$ ($f(n,k)$) denote the maximum possible size of the smallest color class in a (partial) $k$-coloring of the Boolean lattice $B_n$ that does not admit a rainbow antichain of size $k$. The value of $F(n,3)$ and $f(n,2)$ has been recently determined exactly. We prove that for any fixed $k$ if $n$ is large enough, then $F(n,k),f(n,k)=2^{(1/2+o(1))n}$ holds. 
    
    We also introduce the general functions for any poset $P$ and integer $c\ge |P|$: let $F(n,c,P)$ ($f(n,c,P)$) denote the the maximum possible size of the smallest color class in a (partial) $c$-coloring of the Boolean lattice $B_n$ that does not admit a rainbow copy of $P$. We consider the first instances of this general problem.
\end{abstract}

Keywords: Set families, Rainbow Ramsey problems, forbidden subposet problems

\section{Introduction}
In the area of extremal combinatorics, one addresses the problem of finding the largest or smallest structure that possesses a prescribed property. Ramsey-type problems deal with colorings and usually ask for the maximum size of a structure that can be 2-colored (3-colored, 4-colored, $k$-colored) such that a fixed forbidden substructure does not appear in any of the colors (or the forbidden substructure might change from color to color).
In some other coloring problems a rainbow copy of a substructure (a copy all elements of which receive distinct colors) is to be avoided. As rainbow copies can be avoided by simply not using enough many colors, in these kind of problems, one has to pose additional conditions on the coloring.

In this note, we address problems of this last type with respect to set families and inclusion patterns. Let $[n]$ denote the set of the first $n$ positive integers and let $B_n$ be the Boolean lattice of dimension $n$, i.e. the set of elements of $B_n$ is the power set $2^{[n]}$ of $[n]$ ordered by inclusion. For any finite poset $P$ we say a set family $\cG\subseteq B_n$ is a (strong/induced) copy of $P$ if the subposet $B_n[\cG]$ of $B_n$ induced by $\cG$ is isomorphic to $P$, i.e. there exists a bijection $i:P\rightarrow \cG$ such that for any $p,q\in P$ we have $p\le_P q$ if and only if $i(p) \subsetneq i(q)$. If the bijection $i$ satisfies the weaker condition that  $p\le_P q$ implies $i(p) \subsetneq i(q)$, then we say that $\cG$ is a weak / not necessarily induced copy of $P$. A family $\cF$ of sets is induced $P$-free, if it does not contain any induced copy of $P$ and $\cF$ is weak $P$-free if it does not contain a copy of $P$. Forbidden subposet problems ask for the quantity $La^*(n,P)$ ($La(n,P)$) the maximum size of an induced $P$-free (weak $P$-free) family $\cF\subseteq B_n$. This area of extremal combinatorics has been very active since the early 1980's, a recent survey on the topic is \cite{GL}, and the interested reader might also consult the appropriate chapter of the book \cite{GP}. The corresponding Ramsey-type  problems can be formulated as follows: determine the maximum value $N$ for which $B_N$ can be $k$-colored such that the family $\cF_i$ of sets of color $i$ is induced $P_i$-free (weak $P_i$-free) for all $1\le i \le k$. The maximum values are denoted by $R^*(P_1,P_2,\dots, P_k)$ and $R(P_1,P_2,\dots, P_k)$. They were studied recently by Axenovich and Walzer \cite{AW} and Cox and Stolee \cite{CS}. In \cite{tajv}, Chang et al. considered mixed problems: for two posets $P$ and $Q$ what is the maximum dimension $N$ such that $B_N$ can be colored (with as many colors as the painter wants) avoiding a monochromatic induced/weak copy of $P$ in all colors and a rainbow induced/weak copy of $Q$. As an auxiliary problem they introduced the following two functions $F(n,k)$ and $f(n,k)$ as
\begin{itemize}
    \item 
    $F(n,k)$ is the maximum value $m$ such that there exists a $k$-coloring $c:B_n\rightarrow [k]$ that does not admit a rainbow antichain of size $k$ (the poset of $k$ pairwise incomparable elements will be denoted by $A_k$) and all color classes $\cF_i=c^{-1}(\{i\})$ are of size at least $m$,
    \item
    $f(n,k)$ is the maximum value $m$ such that there exists a \textit{partial} $k$-coloring $c:B_n\rightarrow [k]$ that does not admit a rainbow antichain of size $k$ and all color classes $\cF_i=c^{-1}(\{i\})$ are of size at least $m$.
\end{itemize}
By definition, we have $F(n,k)\le f(n,k)$ and the following theorem was proved.

\begin{theorem}[Chang et al \cite{tajv}]\label{A2}
For any even $n\ge 2$ we have $f(n,2)=2^{n/2}-1$, for any odd $n \ge 3$  we have $f(n,2)=2^{\lfloor n/2\rfloor}+1$. Furthermore, if $n$ is large enough, then $F(n,3)=f(n,2)$ holds.
\end{theorem}

In \cite{tajv}, a construction was given to show $(\log_2k-o(1))2^{\lfloor n/2\rfloor}\le f(n,k)\le F(n,k+1)$ thus $\lim_{k\rightarrow \infty}\liminf_{n\rightarrow \infty}\frac{F(n,k)}{2^{n/2}}=\infty$, but no general upper bound was established. The main result of the present paper determines for every fixed $k$ the asymptotics of the exponent of the functions $f(n,k)$ and $F(n,k)$.


\begin{theorem}\label{Ak}
For any $k\ge 2$ there exists $n_0=n_0(k)$ such that if $n\ge n_0$, then we have $F(n,k)\le f(n,k)\le k\cdot 2^{n/2+2\log n\sqrt{n}}$.
\end{theorem}

One can color $B_n$ with more than $k$ colors. Then avoiding a rainbow antichain of size $k$ is even harder. Also, one could be interested in avoiding rainbow strong copies of other posets. So for any positive integer $l$ and finite poset $P$ we define $F(n,l,P)$ to be the maximal value of $m$ such that there exists an $l$-coloring $c:B_n\rightarrow [l]$ that does not admit a strong rainbow copy of $P$ and all color classes of $c$ have size at least $m$. If in the definition we allow partial colorings $c$, then we obtain $f(n,l,P)$ and thus $F(n,l,P)\le f(n,l,P)$ holds for any $l$ and $P$. So the functions $F(n,k)$ and $f(n,k)$ are by definition equal to $F(n,k,A_k)$ and $f(n,k,A_k)$.

It would be natural to introduce the corresponding functions for weak copies of $P$, but instead let us consider forbidding rainbow strong copies of a family $\cP$ of posets. In this way, we obtain the functions $F(n,l,\cP)$ and $f(n,l,\cP)$. Observe that for any poset $P$ we can define $\cP_P=\{P': P'~\text{is a weak copy of}\ P\}$ and then $F(n,l,\cP_P)$ and $f(n,l,\cP_P)$ are just the not necessarily induced versions of $F(n,l,P)$ and $f(n,l,P)$.

Let us remark that by definition for $l<l'$ we have $f(n,l,P)\ge f(n,l',P)$ and $F(n,l,P)\ge F(n,l',P)$ and for any integer $l$ and poset $P$ the inequality $f(n,l,P)\le \lfloor \frac{2^n}{l}\rfloor$ holds trivially.

\begin{problem}
Characterize those posets $P$ for which $f(n,|P|,P)=\lfloor \frac{2^n}{|P|}\rfloor$ holds provided $n$ is large enough.
\end{problem}

By a simple coloring we will show that the diamond poset $D_2$ on four elements $a,b,c,d$ with $a\le b,c\le d$ possesses this property. This might be somewhat surprising to forbidden subposet experts as $D_2$ is the smallest poset $P$ for which the asymptotics of $La(n,P)$ and $La^*(n,P)$ are both unknown.

\medskip

Let us continue with the order of magnitude of $F(n,l,P)$ and $f(n,l,P)$. It turns out that antichains are exceptions. We say that a subset $C$ of a poset $P$ is a component of $P$ if $C$ is maximal with respect to the property that for any $p,q\in C$ there exists a sequence $p_1,p_2,\dots,p_k$ of elements in $C$ such that $p=p_1$, $q=p_k$ and $p_i$ and $p_{i+1}$ are comparable for every $i=1,2,\dots, k-1$. A poset is connected if it has one component. The posets $\vee_k,\wedge_k$ both have $k+1$ elements $a,b_1,b_2,\dots,b_k$ with $a\le_{\vee_k} b_i$ and $b_i\le_{\wedge_k} a$ for all $1\le i \le k$.

\begin{proposition}\label{propi}
\textbf{(i)} For any positive integer $l$ and set $\cP$ of posets that does not contain antichains we have $f(n,l,\cP)\ge 2^{n-m(l)}$, where $m(l)$ is the smallest integer $m$ such that $l\le \binom{m}{\lfloor m/2\rfloor}$ holds.

\textbf{(ii)} Let $l$ be a  positive integer and $\cP$ be  a family of posets such that if $P\in \cP$ has a single component $C$ of size at least 2, then $C$ is not $\vee_k$ nor $\wedge_s$. Then we have $F(n,l,\cP)\ge 2^{n-m(l-1)}$.
\end{proposition}

\begin{proof}
To prove \textbf{(i)} let us fix $l$ sets $S_1,S_2,\dots,S_l\in \binom{[m(l)]}{\lfloor m(l)/2\rfloor}$ and consider the families $\cF_i\subseteq 2^{[n]}$ defined by $\cF_i=\{F\subseteq [n]: F\cap [m]=S_i\}$. As the families $\cF_i$ $i=1,2,\dots,l$ are pairwise incomparable, if sets in $\cF_i$ receive color $i$, then any rainbow system of sets must form an antichain, and therefore there does not exist any rainbow copy of any $P\in \cP$.

Similarly, to prove \textbf{(ii)} let us fix sets $S_1,S_2,\dots, S_{l-1} \in \binom{[m(l-1)]}{\lfloor m(l-1)/2\rfloor}$ and consider the families $\cF_i\subseteq 2^{[n]}$ defined by $\cF_i=\{F\subseteq [n]: F\cap [m]=S_i\}$ for $i=1,2,\dots,l-1$ and $\cF_l=2^{[n]}\setminus \cup_{i=1}^{l-1}\cF_i$. As the families $\cF_i$ $i=1,2,\dots,l$ are pairwise incomparable, a rainbow set of sets must be the disjoint union of an antichain and a $\vee_k$ or of an antichain and a $\wedge_s$. 
\end{proof}

Observe that if we want to avoid a rainbow copy of $\wedge_s$, then as $[n]$ contain all other sets, the other color classes cannot create a rainbow antichain of size $s$. Therefore, by Theorem \ref{Ak}, $F(n,l,\wedge_s)\le f(n,l-1,A_{s})\le 2^{n/2+o(n)}$ holds. 

\begin{proposition}
If $P=\vee_1+A_k$ the disjoint union of a comparable pair and an antichain of size $k$, then for $l\ge k+2$ we have $F(n,l,\vee_1+A_k)\le F(n,k+2,\vee_1+A_k)\le 2^{h(c_0)n+o(n)}$, where $h(x)=-x\log_2x-(1-x)\log_2(1-x)$ is the binary entropy function and $1/3\le c_0\le 1/2$ is the root of the equation $h(x)=(1-x)h(\frac{1-2x}{1-x})$.
\end{proposition}

\begin{proof}
If all color classes of a  $(k+2)$-coloring of $B_n$ has size at least $2^{(h(c_0)+\varepsilon)n}$, then all color classes contain at least $\frac{1}{2}2^{(h(c_0)+\varepsilon)n}$ sets from $\cM:=\{F\subseteq [n]: c_0n\le |F|\le (1-c_0)n\}$. So we can find a comparable pair of sets $F_1$, $F_2$ of different colors (as otherwise all sets in $\cM$ would belong to the same color class). Then we can greedily add the antichain of size $k$: if $M_1\subset M_2$ and $M'_1,M'_2,\dots M'_j$ form a rainbow copy of $\vee_1+A_j$, then as any set $M\in\cM$ is comparable to at most $\sum_{h=0}^{(1-2c_0)n}\binom{(1-c_0)n}{h}=2^{(1-c_0)h(\frac{1-2c_0}{1-c_0})n+o(n)}=2^{h(c_0)n+o(n)}$ other sets of $\cM$, so an unused color class contains at least $\frac{1}{2}2^{(h(c_0)+\varepsilon)n}-j2^{h(c_0)n+o(n)}$ sets from $\cM$ that are incomparable to all $M_1,M_2,M'_1,M'_2,\dots,M'_j$.
\end{proof}

We conjecture that for any poset $P$ to which Proposition \ref{propi} (ii) does not apply, the order of magnitude of $F(n,l,P)$ is less than $2^n$.

\begin{conjecture}
For any $k,s$ and $l\ge k+s+1$ we have $F(n,l,\vee_s+A_k)=F(n,l,\wedge_s+A_k)=o(2^n)$.
\end{conjecture}

\medskip

The most natural non-antichain posets are chains (totally ordered sets). The chain on $k$ elements is denoted by $P_k$. Ahlswede and Zhang \cite{az} proved (in a different context) $f(n,2,P_2)=2^{n-2}$. It is not very hard to see that $f(n,l,P_l)=\lfloor \frac{2^n}{l}\rfloor$ holds for $l\ge 4$. We conjecture $f(n,3,P_3)=2^{n-2}$ for all $n\ge 3$. Moreover, we will present a single coloring that shows $2^{n-2}\le f(n,\vee_2,3),f(n,3,\wedge_3),f(n,3,P_3)$ and prove the following theorem.

\begin{theorem}\label{3}
For any $n\ge 3$, we have $f(n,3,\{\wedge,\vee, P_3\})=2^{n-2}$.
\end{theorem}

The structure of the paper is as follows: in Section 2 we prove Theorem \ref{Ak}, determine $f(n,l,A_2)$ for any $l$ and present a construction for a lower bound on $f(n,l,A_k)$ for general $k$. Section 3 contains the proof of Theorem \ref{3} and all comments and remarks on $F$-functions of non-antichain posets.
\medskip

\textbf{Notation}. For two sets $F,G$ we denote by $[F,G]$ the interval $\{H:F\subseteq H\subseteq G\}$. Similarly, $(F,G]=\{H:F\subsetneq H\subseteq G\}$, $[F,G)=\{H:F\subseteq H\subsetneq G\}$ and $(F,G)=\{H:F\subsetneq H\subsetneq G\}$. For any set $F\subseteq [n]$ we write $\cD_F=[\emptyset, F]$, $\cU_F=[F, [n]]$ and $\cI_F=\cD_F\cup\cU_F$. 

\section{Antichains}

\begin{proof}[Proof of Theorem \ref{Ak}] We proceed by induction on $k$ with the base case $k=2$ being covered by Theorem \ref{A2}. Let $c:B_n\rightarrow \{1,2,\dots,k+1\}$ be a partial $(k+1)$-coloring of $B_n$ that does not admit a rainbow antichain of size $k+1$ and let $\cF_i=\{F:c(F)=i\}$ denote the color classes $i=1,2,\dots,k+1$. Let us define a maximal sequence of $k$-tuples $$(F^1_1,F^1_2,\dots,F^1_k),(F^2_1,F^2_2,\dots,F^2_k),\dots, (F^t_1,F^t_2,\dots,F^t_k)$$ such that 
\begin{itemize}
    \item 
    $F^j_i\in \cF_i\setminus \{F^h_i:h<j\}$, for all $i=1,2,\dots,k$ and $1\le j\le t$,
    \item
    for any $j$ the sets $F^j_1,F^j_2,\dots,F^j_k$ form an antichain of size $k$,
    \item
    $\cup_{i=1}^k\cF_i\setminus \{F^j_i:1\le j\le t, 1\le i \le k\}$ does not contain a rainbow $A_k$.
\end{itemize}  
By the last property and induction we have 
\begin{equation}\label{t}
    \min_{1\le i \le k}|\cF_i|\le f(n,k)+t\le k\cdot 2^{n/2+\log n\sqrt{n}}+t.
\end{equation} 
On the other hand, as $c$ does not admit a rainbow $A_{k+1}$, we must have $\cF_{k+1}\subseteq \cap_{j=1}^t\cup_{i=1}^k\cI_{F^j_i}$. (\ref{t}) implies that if $t\le 2^{n/2+\log n\sqrt{n}}$, then we are done. So suppose $t\ge 2^{n/2+\log n\sqrt{n}}$. Then for any string $\bx=x_1x_2\dots x_a$ of length at most $\sqrt{n}$ with $x_b\in [k]$ for all $1\le b\le a$ we define recursively an index $j_\bx$, a pair $(S_\bx,B_\bx)$ of sets and a downset or an upset $\cM_\bx$ as follows:
\begin{itemize}
    \item 
    for the empty string $\varepsilon$ we have $S_\varepsilon=\emptyset$, $B_\varepsilon=[n]$ and $\cM_\varepsilon=\emptyset$,
    \item
    if $S_\bx \subseteq B_\bx$ and $|B_\bx \setminus S_\bx |\ge n/2+\frac{1}{2}\log n\sqrt{n}$, then  let $j_{\bx }\le t$ be an index such that for any $i\in [k]$ either $|F^{j_{\bx }}_i|\le n/2$ and $|F^{j_{\bx }}_i\setminus S_\bx|\ge \sqrt{n}$ or $|F^{j_{\bx }}_i|\ge n/2$ and $|B_\bx\setminus F^{j_{\bx }}_i|\ge \sqrt{n}$.
    \item
     In the former case, we let $S_{\bx y}:=S_\bx \cup F^{j_{\bx}}_y$, $B_{\bx y}:=B_\bx$, $\cM_{\bx y}:=\cD_{F^{j_{\bx}}_y}$ while in the former case we let $S_{\bx y}:=S_\bx, B_{\bx y}:=B_\bx \cap F^{j_{\bx}}_y$, $\cM_{\bx y}:=\cU_{F^{j_{\bx}}_y}$.
     \item
    if $S_\bx,B_\bx$ are defined and $|B_\bx \setminus S_\bx |\le n/2+\frac{1}{2}\log n\sqrt{n}$ or $S_\bx \not\subseteq B_\bx$, then for any $y\in [k]$ we define $S_{\bx y}:=S_\bx$ and $B_{\bx y}:=B_\bx$,
\end{itemize}

\begin{claim}\label{define}
Whenever $S_\bx \subseteq B_\bx$ and $|B_\bx \setminus S_\bx |\ge n/2+\frac{1}{2}\log n\sqrt{n}$ hold, one can pick an index $j_{\bx }\le t$ with the above properties.
\end{claim}

\begin{proof}[Proof of Claim]
The condition $|B_\bx \setminus S_\bx |\ge n/2+\frac{1}{2}\log n\sqrt{n}$ implies that $|S_\bx|\le n/2-\frac{1}{2}\log n\sqrt{n}$ and $|B_\bx|\ge n/2+\frac{1}{2}\log n\sqrt{n}$ hold. Therefore the number of subsets $G$ with $|G\setminus S_\bx|\le \sqrt{n}$ or $|B_\bx \setminus G|\le \sqrt{n}$ is at most $2\binom{n}{\sqrt{n}}2^{n/2-\frac{1}{2}\log n\sqrt{n}}\le 2^{n/2+\frac{3}{4}\log n\sqrt{n}}$. So the number of indices for which the desired properties do not hold is at most $k\cdot 2^{n/2+\frac{3}{4}\log n\sqrt{n}} < t$, so there exists an index $j_\bx$ as required.
\end{proof}

\begin{claim}\label{cover}
For any $a\le \sqrt{n}$ we have $$\cF_{k+1}\subseteq \bigcup_{|\bx |=a}[S_\bx,B_\bx]\cup \bigcup_{|\bx'|\le a}\cM_{\bx'}.$$
\end{claim}

\begin{proof}[Proof of Claim]
Induction on $a$ with base case $a=0$ being clear as $[S_\varepsilon,B_\varepsilon]=2^{[n]}$. So suppose the statement of the claim is proved for $a$ and let us consider a set $F\in \cF_{k+1}$. If $F$ belongs to $\bigcup_{|\bx'|\le a}\cM_{\bx'}$, then so it does to $\bigcup_{|\bx'|\le a+1}\cM_{\bx'}$. Otherwise $F\in \bigcup_{|\bx |=a}[S_\bx,B_\bx]$ holds, so let $\bx_0$ be a string with $F\in [S_{\bx_0},B_{\bx_0}]$ (in particular, we have $S_{\bx_0}\subseteq B_{\bx_0}$). If $|B_{\bx_0}\setminus S_{\bx_0}|\le n/2+\frac{1}{2}\log n\sqrt{n}$, then $F\in [S_{\bx_0},B_{\bx_0}]=[S_{\bx_0y},B_{\bx_0y}]$ for any $y\in [k]$. Finally, if $|B_{\bx_0}\setminus S_{\bx_0}|\ge n/2+\frac{1}{2}\log n\sqrt{n}$, then by Claim \ref{define} the index $j_{\bx_0}$ is well defined. Therefore, as $c$ does not admit a rainbow $A_{k+1}$, we have $F\in \cI_{F^{j_{\bx_0}}_1}\cup \cI_{F^{j_{\bx_0}}_2} \cup \dots \cup \cI_{F^{j_{\bx_0}}_k}$, and thus for some $y\in [k]$ we must have $F\in \cI_{F^{j_{\bx_0}}_y}$. If either $|F^{j_{\bx_0}}_y|\le n/2$ and $F\subseteq F^{j_{\bx_0}}_y$ or $|F^{j_{\bx_0}}_y|\ge n/2$ and $F\supseteq F^{j_{\bx_0}}_y$, then $F\in \cM_{\bx_0y}$ holds. If either $|F^{j_{\bx_0}}_y|\le n/2$ and $F\supseteq F^{j_{\bx_0}}_y$ or $|F^{j_{\bx_0}}_y|\ge n/2$ and $F\subseteq F^{j_{\bx_0}}_y$, then $F\in [S_{\bx_0y},B_{\bx_0y}]$ holds. This proves the inductive step.
\end{proof}

To bound the size of $\cF_{k+1}$ we use Claim \ref{cover}. The number of strings $\bx$ of length at most $\sqrt{n}$ is not more than $k^{\sqrt{n}+1}$ and each $\cM_{\bx}$ is of size at most $2^{n/2-\frac{1}{2}\log n\sqrt{n}}$, therefore we have $|\bigcup_{|\bx|\le \sqrt{n}}\cM_\bx|\le 2^{n/2-\frac{1}{4}\log n\sqrt{n}}$ if $n$ is large enough. Observe that as long as $S_\bx\subseteq B_\bx$ and the interval does not stabilize, we have $|B_{\bx y}\setminus S_{\bx y}|\le |B_\bx\setminus S_\bx|-\sqrt{n}$ for any string $\bx$ and $y\in [k]$. Therefore, by the time our strings reach the length of $\sqrt{n}$, the intervals stabilize with $|B_\bx\setminus S_\bx|\le n/2+\frac{1}{2}\log n\sqrt{n}$. Thus $|\bigcup_{|\bx |=\sqrt{n}}[S_\bx,B_\bx]|\le k^{\sqrt{n}+1}2^{n/2+\frac{1}{2}\log n\sqrt{n}}\le 2^{n/2+\frac{3}{4}\log n\sqrt{n}}$ holds. According to Claim \ref{cover} we have
\[
|\cF_{k+1}|\le |\bigcup_{|\bx |=\sqrt{n}}[S_\bx,B_\bx]|+|\bigcup_{|\bx|\le \sqrt{n}}\cM_\bx|\le 2^{n/2+\frac{3}{4}\log n\sqrt{n}}+2^{n/2-\frac{1}{4}\log n\sqrt{n}} \le 2^{n/2+\log n\sqrt{n}}.
\]
\end{proof}

\begin{conjecture}
For any integer $k\ge 2$ there exists a constant $C_k$ such that $f(n,k,A_k)\le C_k\cdot 2^{n/2}$ holds.
\end{conjecture}

\begin{construction}\label{congen}
We define a partial $l(k-1)$-coloring $c$ of $B_n$ in the following way such that all color classes have size $2^{n/l+o(n)}$: let us fix $k-1$ chains $\cC_j=\{C^j_1\subset C^j_2 \subset \dots \subset C^j_{l-1}\}$ such that $|C^j_i\setminus C^j_{i-1}|= \frac{n}{l}+o(n)$ for all $1\le j\le k-1$ and $1\le i\le l-1$ with $C^j_0=\emptyset$ for any $j$. We let $C^j_l=[n]$ for all $j$ and for a color $m=(j-1)l+i$ with $1\le j \le k-1$, $1\le i \le l$ we define its color class by
\[
\cF_m=c^{-1}(\{m\})= (C^j_{i-1}, C^j_i] \setminus \bigcup_{j'<j}\bigcup_{h=1}^l(C^{j'}_{h-1}, C^{j'}_h].
\]
Observe that if $F_1,F_2,\dots,F_k$ are colored, then two of them $F_{i_1},F_{i_2}$ are defined using the same chain $\cC_j$ and if they are colored differently, then $F_{i_1},F_{i_2}$ are comparable. Therefore $c$ does not admit a rainbow copy of $A_k$. As for any $j$ and $i$ we have $|(C^j_{i-1},C^j_i]|=2^{ n/l+o(n)}$, all we need to show is that we can choose the chains $\cC_j$ in such a way that  other intervals meet $(C^j_{i-1},C^j_i]$ in $o(2^{n/l})$ sets. First note that it is enough to ensure that $|C^j_i\cap C^{j'}_i|\le (i-1)\frac{n}{l}+\frac{2}{3}\frac{n}{l} $ holds for all $1\le i \le l-1$ and $1\le j\neq j'\le k-1$. Indeed, if this is satisfied, then $|C^j_{i-1}\cup C^{j'}_{i-1}|\ge (i-1)\frac{n}{l}+\frac{1}{3}\frac{n}{l} $ and $|C^j_{i}\cup C^{j'}_{i}|\le (i-1)\frac{n}{l}+\frac{2}{3}\frac{n}{l} $ imply 
$$|(C^j_{i-1},C^j_i] \cap (C^{j'}_{i-1},C^{j'}_i]|=2^{|C^j_i\cap C^{j'}_i|-|C^j_{i-1}\cup C^{j'}_{i-1}|} \le 2^{(i-1)\frac{n}{l}+\frac{2}{3}-((i-1)\frac{n}{l}+\frac{1}{3})}=2^{\frac{1}{3}\frac{n}{l}}.$$
Also, if $i'\neq i$, then $|C^{j'}_{i'}|=\frac{i'}{l}n+o(n)$ implies $|(C^j_{i-1},C^j_i\cap (C^{j'}_{i'-1},C^{j'}_{i'}]|=2^{o(n)}$. 
Therefore $|\cF_m|=|(C^j_{i-1}, C^j_i] \setminus \bigcup_{j'<j}(C^{j'}_{i-1}, C^{j'}_i]|\ge 2^{ n/l-o(n)}-(k-2)2^{\frac{1}{3}\frac{n}{l}}-kl2^{o(n)}.$


Finally, we claim that if the chains $\cC_j$ are generated in the following simple random way, then the condition $|C^j_i\cap C^{j'}_i|\le (i-1)\frac{n}{l}+\frac{2}{3}\frac{n}{l} $ holds for all $1\le i \le l-1$ and $1\le j\neq j'\le k-1$ with probability tending to 1:

\medskip

We let $C^j_0=\emptyset$ for all $1\le j\le k-1$ and set $p_i:=\frac{1}{l-i+1}$ for all $1\le i \le l-1$. Once $C^j_{i-1}$ is defined, then we include every $x\in [n]\setminus C^j_{i-1}$ to $D^j_i$ with probability $p_i$ independently of all other $y\in [n]\setminus C^j_{i-1}$ and let $C^j_i:=D^j_i\cup C^j_{i-1}$.

\medskip

Observe that
\begin{itemize}
    \item 
    $|D^j_i|$ is a binomially distributed random variable $Bi(n-|C^j_{i-1}|,p_i)$,
    \item
    $|[n]\setminus (C^j_i\cup C^{j'}_i)|$ is a binomially distributed random variable $Bi(n,\prod_{h=1}^i(1-p_h)^2)$.
\end{itemize}
So by any correlation inequality (Chernoff, Chebyshev) we obtain that with probability tending to 1, for all $\binom{k-1}{2}(l-1)$ triples $j,j',i$ we have $|C^j_i\cup C^{j'}_i|=(1-\prod_{h=1}^i(1-p_h)^2)n+o(n)$. Similarly, as $p_i(1-\frac{i-1}{l})=\frac{1}{l}$, we obtain that with probability tending to 1, for any pair $j,i$ we have $|C^j_i|=\sum_{h=1}^i|D^j_h|=\frac{i}{l}n+o(n)$. So the condition on the sizes of $C^j_i$'s is satisfied and with probability tending to 1 we have
\[
|C^j_i\cap C^{j'}_i|=\frac{2i}{l}n-\left(1-\prod_{h=1}^i(1-p_h)^2\right)n+o(n)=\left[\frac{2i}{l}+\prod_{h=1}^i\left(\frac{l-h}{l-h+1}\right)^2-1\right]n+o(n).
\]
So we need to show that $\frac{2i}{l}+\prod_{h=1}^i\left(\frac{l-h}{l-h+1}\right)^2-1\le \frac{i}{l}-\frac{1}{3l}$ or equivalently 
\begin{equation}\label{eq}
f(l,i):=\frac{i}{l}+\prod_{h=1}^i\left(\frac{l-h}{l-h+1}\right)^2\le 1-\frac{1}{3l}    
\end{equation}
holds for any $i$ and $l$. Observe that $f(l,i+1)-f(l,i)=\frac{1}{l}+[(\frac{l-i-1}{l-i})^2-1]\prod_{h=1}^i(\frac{l-h}{l-h+1})^2$. Introducing $\Delta_l(i+1):=[1-(\frac{l-i-1}{l-i})^2]\prod_{h=1}^i(\frac{l-h}{l-h+1})^2$, we can see that
\[
\frac{\Delta_l(i+1)}{\Delta_l(i)}=\frac{1-(\frac{l-i-1}{l-i})^2}{1-(\frac{l-i-2}{l-i-1})^2}\left(\frac{l-i}{l-i+1}\right)^2<1.
\]
This shows that $\Delta_l(i)$ is decreasing in $i$ and therefore $f(l,i)$ is convex in $i$ so it takes its maximum either at $i=1$ or at $i=l-1$. The right hand side of (\ref{eq}) is constant in $i$, so it is enough to check if $f(l,1)$ and $f(l,l-1)$ are both at most $1-\frac{1}{3l}$. We have $$f(l,1)=\frac{1}{l}+\left(\frac{l-1}{l}\right)^2=\frac{l^2-l+1}{l^2}<\frac{l^2-l/3}{l^2}=1-\frac{1}{3l}.$$
For $f(l,l-1)=\frac{l-1}{l}+\prod_{h=1}^{l-1}(\frac{l-h}{l-h+1})^2\le 1-\frac{1}{3l}$ we need $g(l):=\prod_{h=1}^{l-1}(\frac{l-h}{l-h+1})^2\le \frac{2}{3l}$. This holds true for $l=2$. As $g(l+1)=\frac{l^2}{(l+1)^2}g(l)$, we see that $\frac{g(l+1)}{g(l)}=\frac{l^2}{(l+1)^2}<\frac{l}{l+1}=\frac{\frac{2}{3(l+1)}}{\frac{2}{3l}}$, we obtain that $g(l)$ decreases quicker in $l$ than $\frac{2}{3l}$, so our required inequality holds for all $l\ge 2$.
\end{construction}

The conjecture below states that for any fixed $k$ and $l$ Construction \ref{congen} is not far from being optimal.

\begin{conjecture}
For any integers $(l-1)(k-1)<c\le l(k-1)$ we have $f(n,c,A_k)=2^{(1/l+o(1))n}$.
\end{conjecture}

We end this section by determining the value of $f(n,c,A_2)$ for all $n$ and $c$. We will use the following lemma first proved by Ahlswede and Zhang \cite{az} that appeared in this form in \cite{tajv}.

\begin{lemma}\label{azlemma}
Let $\cF_1,\cF_2,\dots, \cF_m\subseteq 2^{[n]}$ be families such that for any $1\le i\neq j \le m$ and $F_i\in \cF_i,F_j \in \cF_j$ the sets $F_i$ and $F_j$ are comparable. Then there exists a chain $\cC=\{\emptyset=C_0\subsetneq C_1 \subsetneq \dots \subsetneq C_t=[n] \}$ such that the set $[t]=\{1,2,\dots,t\}$ can be partitioned into $m$ sets $T_1,T_2,\dots,T_m$ with $\cF_i\subseteq \cC\cup \bigcup_{h\in T_i}(C_{h-1},C_h)$.
\end{lemma}

\begin{theorem}
Let $l \log_2l\le n$ be positive integers and let $a$ be the integer with $1\le a\le l$ and $l-a\equiv n ~(\mod l)$. Then $f(n,l,A_2)=2^{\lfloor n/l\rfloor}-2+\lfloor \frac{l+1}{a}\rfloor$ holds.
\end{theorem}

\begin{proof}
First observe that $a$ is the number of parts of size $\lfloor\frac{n}{l}\rfloor$ in an equipartition of $[n]$ into $l$ parts. Let us consider the coloring showing $f(n,l,A_2)\ge 2^{\lfloor n/l\rfloor}-2+\lfloor \frac{l+1}{a}\rfloor$. Let $\emptyset=C_0\subset C_1\subset \dots C_{l-1}\subset C_l=[n]$ such that $|C_i\setminus C_{i-1}|=\lfloor \frac{n+i-1}{l}\rfloor$ holds for all $i=1,2,\dots,l$. According to the previous observation the first $a$ of these sets have size $\lfloor\frac{n}{l}\rfloor$, the others $\lfloor\frac{n}{l}\rfloor+1$. So if we let $c(H)=i$ if $H \in (C_{i-1},C_i)$ and distribute the $l+1$ $C_j$'s among the $a$ small color classes evenly, then the smallest color classes will have size $2^{\lfloor n/l\rfloor}-2+\lfloor \frac{l+1}{a}\rfloor$ as required.

To see the upper bound let $c$ be a partial $l$-coloring of $B_n$ that does not admit a rainbow pair of incomparable sets. Then the color classes $\cF_i=c^{-1}(\{i\})$ ($i=1,2,\dots,l$) satisfy the conditions of Lemma \ref{azlemma}, so let the chain $\cC=\{\emptyset=C_0\subsetneq C_1 \subsetneq \dots \subsetneq C_t=[n]\}$ and the sets $T_1,T_2,\dots,T_l$ of the partition of $[t]$ as in the lemma. Observe that $|\cF_i|=\sum_{h\in T_i}(2^{|C_h\setminus C_{h-1}|}-2)+c_i$ where $c_i$ is the number of sets in $\cC$ with color $i$, so $\sum_{i=1}^lc_i=t+1$. Our aim is to apply some trasformations to $\cC$ such that the correspnding  new colorings' smallest color class size does not decrease and finally we obtain the coloring of the first paragraph. First, $\cC$ can be changed such that the $T_i$'s consist of consecutive elements of $[t]$. Indeed, we can create $\cC'=\{\emptyset=C'_0\subsetneq C'_1 \subsetneq \dots \subsetneq C'_t=[n]\}$ such that if $T_1=\{h_1,h_2,\dots, h_{s_1}\}$, then $T'_1=\{1,2,\dots,s_1\}$ and $|C'_j\setminus C'_{j-1}|=|C_{h_j}\setminus C_{h_j-1}|$ and the color of $C_h$ equals the color of $C'_h$. 

Next we can assume that if $c_i>0$, then $T'_i$ is a singleton. Indeed, if not, then $T'_i$ contains $h,h+1$ for some $h$ and we can assume that the color of $C'_h$ is $i$ (maybe after exchanging the colors of $C'_h$ and that set of $\cC'$ that was colored $i$). Then removing $C'_h$ from the chain strictly increases the color class $i$ as $2^a-2+2^b+1<2^{a+b}-2$ holds for all positive integers $a,b$. Similarly, if $c_i=0$, then with the exception of at most one $h\in T_i$, we have $|C'_{h}\setminus C'_{h-1}|=1$ (if $h,h+1\in T_i$ with $|C'_h\setminus C'_{h-1}|,|C'_{h+1}\setminus C'_h|\ge 2$, then we can change $C'_h$ to have size $|C'_{h-1}|+1$ without changing the color $C'_h$ and strictly increasing the size of the color class $\cF_i$). 

So far we have obtained that color classes containing some $C_h$'s have one interval $(C_{h-1},C_h)$, while those not containing any elements of $\cC$ can have one large interval and possibly some others of dimension 1. But observe that in this latter case, if $h,h+1\in T_i$ with $|C_h\setminus C_{h-1}|=1$ and $c(C_h)=j,c(C_{h-1})=j'$, then $(C_{h-1},C_h)$ is empty, so $C_{h-1}$ can be removed from $\cC$ and an extra 1 can be added to the dimension of the interval belonging to color $j'$. This increases the color class of $j'$ and does not change the size of any other color classes. With these changes one make sure that all $T_i$'s are singletons, i.e. $t=l$. Suppose we have a color class, say color 1, the interval of which has dimension strictly smaller than $\lfloor n/l\rfloor$. Then there is another color class, say color 2, the interval of which has dimension at least $\lfloor n/l\rfloor+1$. Then to have $|\cF_1|\ge 2^{\lfloor n/l\rfloor}+\lfloor l/a\rfloor$, the color class $\cF_1$ must contain at least $2^{\lfloor n/l\rfloor-1}+\lfloor l/a\rfloor$ sets of $\cC$. The assumption $l\log_2 l\ge n$ implies $2^{\lfloor n/l\rfloor-1}\ge \lfloor l/a\rfloor$, so decreasing the dimension of the interval of $\cF_2$ and increasing the interval of $\cF_1$ and possibly recoloring $\lfloor l/a\rfloor$ sets of $\cC$ from color 1 to color 2 will yield an even better coloring. So we can assume that all intervals have dimension at least $\lfloor l/a\rfloor$. The minimum number of these colors is $a$, so if we distribute the $l+1$ sets of $\cC$ among them evenly, the best we can get is $2^{\lfloor n/l\rfloor}-2+\lfloor \frac{l+1}{a}\rfloor$ as claimed.
\end{proof}
\section{Other posets}

Among non-antichain posets let us consider first chains.
First observe that if $c$ is a total $l$-coloring of $B_n$ that does not admit a rainbow copy of $P_k$ and $c(\emptyset)=i$, then the partial coloring $c'$ obtained from $c$ by omitting the color class $\cF_i$ does not admit a rainbow copy of $P_{k-1}$, so we have $F(n,l,P_k)\le f(n,l-1,P_{k-1})$. $F(n,2,P_2)=0$ as if $c$ does not admit a rainbow $P_2$, then all sets must share the color of $\emptyset$. By the above observation $F(n,3,P_3)\le f(n,2,P_2)$. Ahlswede and Zhang proved \cite{az} that the latter equals $2^{n-2}$ and the following construction shows $F(n,3,P_3)=2^{n-2}$: $c(F)=1$ if $1\in F,2\notin F$, $c(F)=2$ if $1\notin F,2\in F$, $c(F)=3$ otherwise. As mentioned above Ahlswede and Zhang proved $f(n,2,P_2)=2^{n-2}$. They considered families $\cF_1,\cF_2,\dots,\cF_l$ with the property that for any $F_i\in \cF_i$ and $F_j\in \cF_j$ with $i\neq j$ the sets $F_i$ and $F_j$ are incomparable. They called these families cloud antichains, later Gerbner et al \cite{cross} studied them under the name of cross-Sperner families. The upper bound on $f(n,2,P_2)$ follows from the following theorem.

\begin{theorem}[Ahlswede, Zhang \cite{az}, Gerbner et al \cite{cross}]\label{prod}
If $\cF_1,\cF_2 \subseteq 2^{[n]}$ are families such that  any  pair $F_1\in \cF_1, F_2 \in \cF_2$ is incomparable, then $|\cF_1||\cF_2|\le 2^{2n-4}$. In paricular, $\min\{|\cF_1|,|\cF_2|\}\le 2^{n-2}.$
\end{theorem}

If $k\ge 4$, then by definition we have $f(n,k,P_k)\le \lfloor \frac{2^n}{k}\rfloor$ and considering two families $\cF_1\subseteq \{F:1\in F,2\notin F\}$, $\cF_2\subseteq \{F:1\notin F,2\in F\}$ with $|\cF_1|=|\cF_2|=\lfloor \frac{2^n}{k}\rfloor$ and an arbitrary coloring of the remaining sets with equal color classes shows $f(n,k,P_k)=\lfloor \frac{2^n}{k}\rfloor$. So the only value for which $f(n,k,P_k)$ is unknown is $k=3$. By the above we have $2^{n-2}\le f(n,3,P_3)\le \lfloor \frac{2^n}{3}\rfloor$ and we conjecture the lower bound to be tight. Furthermore, we also conjecture that $f(n,3,\vee_2)=f(n,3,\wedge_2)=2^{n-2}$ holds. The following proposition gives colorings showing the lower bound of this conjecture.

 \begin{proposition}
 For any $n\ge 3$ we have 

\textbf{(i)} $f(n,3,\{P_3,\vee_2,\wedge_2\})\ge 2^{n-2}$,
 
 \textbf{(ii)} $F(n,4,D_2)=f(n,4,D_2)=2^{n-2}$.
 \end{proposition}

\begin{proof}
Let us define first a 4-coloring $c'$ of $2^{[3]}$ by letting $c'(\{1\})=c'(\{1,2\})=1, c'(\{2\})=c'(\{2,3\})=2, c'(\{3\})=c'(\{1,3\})=3, c'(\emptyset)=c'([3])=4$. Let us then write $\cF_i=\{F\in 2^{[n]}: c'(F\cap [3])=i\}$. The 3-coloring with color classes $\cF_1,\cF_2,\cF_3$ does not admit rainbow copies of $P_3,\wedge_2$ and $\vee_2$ as $c'(\{i\})=c'(\{i,i+1\})$ where addition is modulo 3. This also implies that the 4-coloring with color classes $\cF_1,\cF_2,\cF_3,\cF_4$ does not admit a rainbow copy of $D_2$.
\end{proof} 

Now we prove Theorem \ref{3} that states if we forbid rainbow copies of $P_3,\vee_2,\wedge_2$ simultaneously, then the above construction gives the value of $f(n,3,\{P_3,\wedge_2,\vee_2\})$.

\begin{proof}[Proof of Theorem \ref{3}]
Let $c:B_n\rightarrow [3]$ be a 3-coloring that avoids rainbow copies of $\wedge_2, \vee_2$ and $P_3$. For $1\le i\neq j \le 3$ we define $\cF^j_i:=\{F:c(F)=i, \exists G \hskip 0.3truecm c(G)=j, F ~\text{is comparable to}\ G\}$. Let us observe that
\begin{enumerate}
    \item 
    for any distinct $i,j,k$ we have $\cF_i^j\cap \cF_i^{k}=\emptyset$ as if $c(F)=i$ is comparable to $G_1$ and $G_2$ with $c(G_1)=j, c(G_2)=k$, then $F,G_1,G_2$ form a rainbow copy of either $\wedge_2$ or $\vee_2$ or $P_3$,
    \item
    for any distinct $i,j,k$ the families $(\cF_i\setminus \cF^k_i)\cup (\cF_j\setminus \cF^k_j)$ and $\cF_k$ are cross-Sperner by definition.
\end{enumerate}
The latter observation and Theorem \ref{prod} imply
\[
(|\cF_i\setminus \cF^k_i|+ |\cF_j\setminus \cF^k_j|)\cdot |\cF_k|\le 2^{2n-4}
\]
for any distinct $i,j$ and $k$. If $|\cF_k|\le 2^{n-2}$ for some $k=1,2,3$, then we are done. Otherwise for any pair $1\le i\neq j\le 3$ we have
\[
|\cF_i\setminus \cF^k_i|+ |\cF_j\setminus \cF^k_j|\le 2^{n-2}.
\]
Summing this for all three pairs $i,j$ and applying the first observation above we obtain
\[
\sum_{i=1,2,3}|\cF_i|\le \sum_{k=1,2,3}(|\cF_i\setminus \cF^k_i|+ |\cF_j\setminus \cF^k_j|)\le 3\cdot 2^{n-2}.
\]
This implies that at least one of the $\cF_i$'s have size at most $2^{n-2}$.
\end{proof}

\bibliographystyle{acm}
\bibliography{refs_patkos}
\end{document}